\documentclass[graybox]{svmult}
\pdfoutput=1

\usepackage{mathptmx}       % selects Times Roman as basic font
\usepackage{helvet}         % selects Helvetica as sans-serif font
\usepackage{courier}        % selects Courier as typewriter font
\usepackage{type1cm}        % activate if the above 3 fonts are
\usepackage{makeidx}         % allows index generation
\usepackage{graphicx}        % standard LaTeX graphics tool
\usepackage{multicol}        % used for the two-column index
\usepackage[bottom]{footmisc}% places footnotes at page bottom
\usepackage{changes}
\makeindex             % used for the subject index
\usepackage{amsmath,amssymb}
\usepackage{amsfonts}

\newcommand{\Ad}{\mathop{\mathrm{Ad}}\nolimits}
\newcommand{\ad}{\mathop{\mathrm{ad}}\nolimits}

\newcommand{\Jac}{\mathrm{Jac}}

\newcommand{\R}{\mathbb{R}}

\newcommand{\dif}{\mathrm{d}}

\newcommand\Diff{\textit{Diff}}

\def\1{\mathchoice {\rm 1\mskip-4mu l} {\rm 1\mskip-4mu l}
{\rm 1\mskip-4.5mu l} {\rm 1\mskip-5mu l}}
\newif\iftodo
\todotrue

\iftodo
\newcommand{\todo}[1]{\vspace{5 mm}\par \noindent
\marginpar{\textrm{ToDo}} \framebox{\begin{minipage}[c]{0.45
\textwidth} \tt #1 \end{minipage}}\vspace{5 mm}\par}
\else
\newcommand{\todo}[1]{}
\fi

\begin{document}
%\title*{Left-invariant metrics on diffeomorphism groups for image matching}
\title*{Diffeomorphic image matching with left-invariant metrics}
\titlerunning{Left-invariant diffeomorphic matching}
\author{Tanya Schmah and Laurent Risser and Fran\c{c}ois-Xavier Vialard\thanks{All authors contributed equally to this work.}}
\institute{Tanya Schmah \at 
Rotman Research Institute, Baycrest, and University of Toronto, \\
\email{tschmah@research.baycrest.org
}
\and Laurent Risser \at CNRS - Institut de
Math\'ematiques de Toulouse, \email{lrisser@math.univ-toulouse.fr}
\and Fran\c{c}ois-Xavier Vialard \at Universit\'e Paris-Dauphine, \email{vialard@ceremade.dauphine.fr}}
\maketitle

\abstract{
The geometric approach to diffeomorphic image registration known as
\textit{large deformation by diffeomorphic metric mapping} (LDDMM)
is based on a left action of diffeomorphisms on images, 
and a right-invariant metric on a diffeomorphism group,
usually defined using a reproducing kernel.
We explore the use of left-invariant metrics on diffeomorphism groups, based on reproducing kernels
defined in the body coordinates of a source image.
This perspective, which we call Left-LDM, allows us to consider non-isotropic spatially-varying kernels,
which can be interpreted as describing variable deformability of the source image.
We also show a simple relationship between LDDMM and the new approach,
implying that spatially-varying kernels are interpretable in the same way in LDDMM. 
We conclude with a discussion of a class of kernels that enforce a soft mirror-symmetry constraint,
which we validate in numerical experiments on a model of a lesioned brain.
}

\section{Introduction}
The geometric point of view on diffeomorphic image matching 
was pioneered by \cite{DuGrMi1998,Trouve1998}, and has been developed
significantly in the last ten years \cite{Begetal2005,TrouvŽ05metamorphosesthrough,Holm08theeuler,HolmSolitons,CotterClebsch,
MomentumImagesBruveris,Gay-Balmaz11optim}.
%, with
%an extensive overview appearing in \cite{Gay-Balmaz11optim}.
In its many practical applications to medical imaging, including computational anatomy 
\cite{Miller09_CFA},
the approach is known as the Large Deformation Diffeomorphic Metric Mapping framework (LDDMM).
A good geometric overview may be found in  \cite{MomentumImagesBruveris}.
%%new models have been introduced such as metamorphoses in \cite{TrouvŽ05metamorphosesthrough}. 
 Two key elements of this framework are: a right-invariant Riemannian metric on a group of diffeomorphisms;
and the left action of this group on images $I:\Omega \to \R^d$ defined by 
$\phi \cdot I := I \circ \phi^{-1}$.
Combining these two elements gives an induced Riemannian metric on the group orbit
of a given image $I$. 
 
In image registration in general, the \textit{inexact matching problem} is, given two images $I$ and $J$,
to find a transformation $\phi$  that minimises the sum of some measure of the size of $\phi$
and some measure of image dissimilarity (or error) $E(\phi \cdot I, J)$,
such as  $\| \phi \cdot I - J\|^2_{L^2}$.
In LDDMM, we seek a path of diffeomorphisms $\phi(t)$ starting at $Id$, with the size of the final diffeomorphism $\phi(1)$ given by the length of the path $\phi$ defined
by the right-invariant Riemannian metric associated with some norm $\| . \|_{V}$ on a Hilbert space $V$ of smooth vector fields.
Thus the fundamental optimisation problem in LDDMM is to minimise
\begin{equation} \label{LDDMM}
\mathcal{J}(\phi) = \frac 12 \int_0^1 \! \|v(t)\|_{V}^2 \, \dif t +
%\frac{\lambda}{2} 
E(\phi(1) \cdot I, J),
%\frac{\lambda}{2} \| \phi(1) \cdot I - J\|^2_{L^2}\,,
\end{equation}
for a path $\phi$ with $\phi(0) = Id$,
under the constraint
\begin{equation} \label{spatialvel}
\partial_t \phi(t) = v(t) \circ \phi(t),
\end{equation}
which defines $v(t)$ as the \textit{spatial (Eulerian) velocity} of $\phi(t)$.
%In general, the second term may be substituted for another ``matching'' term.
Note that all minimisers of this functional are geodesics, since they must minimise the first
term of  \eqref{LDDMM} for a given $\phi(1)$.
%The most widely-used solution to this optimisation problem is the variational method developed in %\cite{DuGrMi1998,Trouve1998,Begetal2005}, which is a gradient descent algorithm
%based on the Euler-Lagrange equations characterising the stationary points of \eqref{LDDMM}.

The minimisation problem \eqref{LDDMM} is well-posed provided that the norm on $V$ is sufficiently strong in terms of smoothness (see \cite{laurentbook}, Theorem 11.2).
The Hilbert space $V$ is usually defined via its reproducing kernel:
\begin{equation} \label{norm}
%\|v\|^2_{K_\sigma} = \langle p, K_\sigma \star p\rangle_{L^2} \text{ where } v= K_\sigma \star p \,.
\|v\|^2_{K} = \langle p, K \star p\rangle_{L^2}, \text{ where } v= K \star p \,.
\end{equation}
A Gaussian kernel is often chosen for computational convenience, or a mixture of Gaussian kernels as in \cite{RisserMICCAI,Risser11TMI}.

We note that LDDMM is not the only diffeomorphism-based approach to image matching.
There is another family of successful methods, based on exponentiating stationary vector fields 
\cite{Arsigny06, Vercauteren09, Ashburner07}.
However, unlike these methods, LDDMM is able to draw on concepts in geometry
and mechanics such as geodesic distance and momentum, which have been central both to theoretical developments and to recent efficient numerical algorithms \cite{GeodesicShootingVialard,ParticleMesh}. 

Though not required by the theory, in practice the kernel used in diffeomorphic methods
(LDDMM and the other methods cited above)
has always been
chosen to be translationally-invariant and isotropic.
%$  and rotationally invariant, i.e. spatially-constant and isotropic. 
%In fact, we are not aware of any image registration method, including free-form deformation methods,
%that uses spatially-varying or non-isotropic regularisation.
In LDDMM, spatially-varying or non-isotropic (``direction-dependent'')
kernels have no obvious interpretation, because
the norm is defined in Eulerian coordinates, so that
as $t$ varies during the deformation, a fixed point in the
source image moves through space, and conversely, a fixed point in
space will correspond to different points on the source image.
Similarly, the directions in a direction-dependent kernel are defined 
with respect to Eulerian coordinates,
not the coordinates of the moving source image.
Nonetheless, spatially-varying kernels are potentially of great interest
in medical applications, if they can be made to represent spatially-variable (or non-isotropic) 
deformability of tissue. This is indeed already done in \cite{PcwDiffRisser} to model sliding conditions between the lungs and the ribs.
%This has natural interpretations
%in terms of either physical tissue properties (for intra-subject registration),
%growth and development of normal or diseased tissue (for longitudinal studies),
%or anatomical variability (for inter-subject registration).
In general it is well-known that a good choice of kernel (the ``regulariser'') is essential for optimising registration performance,
so that taking into account any spatial variability of the tissue deformability in the kernel will improve the registration.
%A recent example of this is the direction-dependent regularisation in the free-form deformation context in
%\cite{Schmidt-Richberg12}, which uses a fixed regularisation scheme based on prior anatomical knowledge.
%Automatic optimisation of regularisers is another promising avenue, especially when giving interpretable information in the parameter values.
% Tanya you CAN KEEP THAT FOR THE DISCUSSION
%In \cite{SimpsonNI2012},
%the optimised parameter is global, but similar methods could be developed 
%for spatially-varying regularisation.

%Such methods not only have the potential for improved registration, 
%but would give as a by-product interpretable information in the optimal parameter values,
%which would describe which areas of the tissue are most deformable.

With this motivation, we propose a new registration framework, which will support 
natural interpretations of spatially-varying metrics.
Left-Invariant LDDMM (``Left-LDM'')
%Diffeomorphic Metric Mapping (Left-LDM)
% (Left-LDM?) 
is analogous
to LDDMM but based on a \textit{left-}invariant metric, i.e. based on a norm in
the body (Lagrangian) coordinates of the source image.
This means that instead of the norm in \eqref{LDDMM} being applied to the spatial
(Eulerian) velocity defined
by \eqref{spatialvel},
it is applied to the \textit{convective velocity} defined by
\begin{equation}\label{convel}
\partial_t \phi(t) = d\phi(t)\cdot v(t)\, ,
\end{equation}
where $d\phi(t)$ is the spatial derivative of $\phi(t)$.
To emphasize the relationship between the two frameworks, we will refer to LDDMM from now on
as ``Right-LDM'', consistent with the use of the shortened acronym LDM in \cite{Gay-Balmaz11optim}.
The matching problem in Left-LDM is to minimize the same functional as in Right-LDM
\eqref{LDDMM} but under the "new" constraint \eqref{convel}.
%where $d\phi(t)\cdot v(t)$ means that the flow $\phi(t) \cdot x$ of
%each point $x$ is not computed by integrating $v(t)$ along the flow
%as in \eqref{spatialvel} but along $x$ instead. {\bf (Tanya, don't
%hesitate to tell me if I'm wrong.)}
%The  matching problem in LIDM is to minimize the same functional as in LDDMM
%\eqref{LDDMM} but under the "new" constraint \eqref{convel}.
Note that the convective velocity of a given $\phi(t)$ is the 
pull-back of the spatial velocity by $\phi_t$, i.e. it is just the spatial
velocity expressed in body (Lagrangian) coordinates.

Subject to some analytical subtleties explored in Section \ref{sect:analysis},
the solutions $\phi(t)$ are left-geodesics in a diffeomorphism group.
%The solutions satisfy Euler-Poincar\'e equations stated in Section \ref{sect:geodesicflow}.
The description of left-geodesic flow in terms of the convective velocity 
is an example of a convective representation of a continuum theory.
Convective representations were introduced in \cite{HoMaRa1986} for ideal fluid flow,
and \cite{SiMaKr1988} for elasticity, and the subject has been further developed in 
\cite{GBMaRa2012}.
The relationship between left- and right- geodesic flows on a diffeomorphism group was 
explored earlier in \cite{GaRa11Clebsch}.

In the Left-LDM framework, a spatially-varying or non-isotropic kernel makes sense,
because it is defined in Lagrangian coordinates, so it can
model variable deformability of different parts of the source image.
(The norm is carried along by push-forwards with the moving source image.)
This opens up possibilities for application-specific regularisation,
either hand-tuned or learnt from data.

%The optimality equations are very similar in their appearance but the
%expected behaviour is very different since the property of the metric
%(say divergence free vector fields for instance) is preserved in
%Lagrangian coordinates.
\iffalse
Section \ref{GradientSection} presents the gradient calculation of the matching functional in the LIDM model and Section \ref{relation} develops the correspondence between LIDM and LDDMM. Section \ref{PulsonsSection} presents finite dimensional parametrizations of a class of LIDM optimal paths, and Section \ref{sec:Evaluation} shows the performance of the LIDM model on synthetic and real data. Finally, Section \ref{DiscussionSection} concludes the paper and discusses research perspectives raised by the LIDM model.
\fi

\section{Analytical setting}\label{sect:analysis}

We consider the convective velocity constraint, 
formula \eqref{convel}, and the conditions on $v(t)$ such that it can be integrated to produce 
the diffeomorphism $\phi(t)$. Such an evolution equation is a partial differential equation that belongs to the class of linear symmetric hyperbolic systems \cite{FischerMarsden}. The usual method for solving such equations consists in using the method of characteristics, which amounts to solve an equation of the type \eqref{spatialvel} on the inverse of the flow. The equation of characteristics, being equivalent to  formula \eqref{spatialvel},
is an ordinary differential equation and can be integrated provided sufficient smoothness assumptions on the \textit{spatial velocity}. For the spatial velocity constraint, a satisfactory answer has been given in
\cite[Theorems 8.7 and 8.14]{laurentbook}:
% note: $L^2([0,1],V) = \mathcal{X}_V^2 \subset \mathcal{X}_V^1$, the latter being in Thm 8.14
The flow of a time dependent vector field in $L^2([0,1],V)$ is well-defined if there exists a constant $C>0$ such that for every $v \in V$
\begin{equation}\label{Admissible}
\| v \|_{1,\infty} \leq C \| v \|_V\,,
\end{equation}
where $\| v \|_{1,\infty}$ is the Banach norm in $W^{1,\infty}(\Omega,\R^d)$. Under this hypothesis, the variational problem \eqref{LDDMM} is well-posed and the set $G_R$,
defined by\footnote{In the corresponding definition in \cite{laurentbook}, $v$ need only be absolutely integrable in time.}
\begin{equation}
G_R := \left\{ \phi(1) \, \big| \, \partial_t \phi(t) = v(t)\circ \phi(t) \text{ and } v \in L^2([0,1],V)\,, \phi(0)=Id\, \right\}\,,
\end{equation}
is a group. 
A similar approach in \cite{em70} proves that the flow of $v \in C([0,1],H^s)$ defines an $H^s$ diffeomorphisms for $s> d/2+ 2$. From a variational point of view the former approach is better suited for solving Problem \eqref{LDDMM}. In particular, working with the space $L^2([0,1],V)$ is crucial for proving the existence of a minimizer and therefore we cannot  reduce our work to a smooth setting. This is our main motivation for developing the following analytical study. Let us then define the following set,
\begin{equation}
G_L := \left\{ \phi(1) \, \big| \, \partial_t \phi(t) = d\phi(t)\cdot v(t) \text{ and } v \in L^2([0,1],V)\,, \phi(0)=Id\,  \right\} \,.
\end{equation} 
Integrating equation \eqref{convel} is straightforward in a smooth setting. Indeed, this equation is equivalent to \begin{equation}\label{convelinv2}
\partial_t \phi^{-1}(t) = -v(t) \circ \phi^{-1}(t)\,.
\end{equation}
Unfortunately, working with $L^2([0,1],V)$ vector fields, Equation \eqref{convelinv2} has to be proven true in that context.
An example of this issue is the following: with a fixed regularity, for instance the group $\Diff^s$ of $H^s$ diffeomorphisms, the inversion map is only continuous and not differentiable. This comes from the fact that the inversion map $Inv: \Diff^s \to \Diff^{s}$ presents a loss of regularity when being differentiated:
\begin{equation}
D\;Inv(\phi)(v) = -d\phi^{-1}(v \circ \phi^{-1}) \,.
\end{equation}
The rest of the section will be devoted to show that equation \eqref{convel} can be solved via the method of characteristics. Our strategy consists in proving that Equation \eqref{convelinv2} holds under very weak conditions so that integration of the \textit{convective velocity} equation \eqref{convel} reduces to the integration of Equation \eqref{convelinv2}.

In what follows, we consider $\Omega$ a closed, bounded domain and $V$ a Hilbert space of vector fields $u$ such that both $u$ and $du$ vanish on its boundary, and we suppose that $V$ is embedded in
$C^1(\Omega, \R^d)$, i.e. there exists a constant $C>0$ such that \eqref{Admissible} applies for all $u$.
Let us begin with the following lemma:
\begin{lemma}
Let $B := C^0([0,1],C^1_\infty(\Omega,\R^d)) \cap H^1([0,1],L^2(\Omega,\R^d))$.
Let $\phi \in B$, and denote by $\phi^{-1}$ the map $t \mapsto \phi_t^{-1}$.
If $\phi_t$ is a diffeomorphism onto $\Omega$ for all $t \in [0,1]$, then $\phi^{-1}$ lies in $B$.
\end{lemma}

\begin{remark} The subscript $\infty$ denotes the use of the sup norm. 
%For example, $C^1_\infty(\Omega,\R^d)$ is the space of $C^1$ maps from $\Omega$ into $\R^d$ endowed with the norm $\| f \|_{\infty} = \sup_{x \in \Omega} \|  f(x) \|$.
\end{remark}

\begin{proof}
The standard Inverse Function Theorem implies that $\phi_t^{-1} $ is $C^1$ for all $t \in [0,1]$.
The continuity of $\phi$ implies the continuity of the map $(t,x) \mapsto \phi_t(x)$, 
which by a lesser-known version of the Implicit Function Theorem (see \cite{ImplicitSpringerEOM})
implies the continuity of $t\mapsto \phi_t^{-1}(x)$ for every $x$.
%The continuity with respect to the parameter $t$ is given by the continuity of the fixed point solution of contractive mappings with a continuous parameter. 
Therefore, by compactness of $\Omega$ we have  $\phi^{-1} \in C^0([0,1],C^1_\infty(\Omega, \R^d)$.
 
Let us first suppose that $\phi \in C:= C^0([0,1],C^1_\infty(\Omega, \R^d) \cap C^1([0,1],C^0(\Omega,\R^d))$, and that (as before) $\phi_t$ is a diffeomorphism onto $\Omega$ for all $t \in [0,1]$.
Then for all $x\in \Omega$ one has by simple differentiation
\begin{equation}
\partial_t \phi^{-1}_t(x) = -[d\phi_t]_{\phi_t^{-1}(x)}(\partial_t \phi_t(\phi_t^{-1}(x)))\,.
\end{equation}
We aim at proving that $\partial_t \phi^{-1}_t$ belongs to $L^2([0,1],L^2(\Omega,\R^d))$: The first term $[d\phi_t]_{\phi_t^{-1}(x)}$ is continuous (on $\Omega$) and its sup norm is uniformly bounded for $t \in [0,1]$ since $C^0_{\infty}([0,1],C^1_\infty(\Omega,\R^d))$.
By assumption, $\partial_t \phi_t \in L^2(\Omega,\R^d)$ and the right composition with a $C^1$ diffeomorphism is a bounded linear operator on $L^2(\Omega,\R^d)$ (by a standard change of variable). It follows easily that $\partial_t \phi^{-1}_t \in L^2([0,1],L^2(\Omega,\R^d))$ and $\phi^{-1} \in C$.

We will prove a similar result for any $\phi \in B$: By density of $C$ in $B$, we consider a sequence $\phi_n \in C$ converging to $\phi \in B$. In particular, we have
\begin{equation}
\phi_{n,T}^{-1}(x) = \int_0^T\! -[d\phi_{n,t}]_{\phi_{n,t}^{-1}(x)}(\partial_t \phi_{n,t}(\phi_{n,t}^{-1}(x)))\, \dif t \,.
\end{equation}
First, the left-hand side strongly converges in $C^1_\infty(\Omega,\R^d)$  (by the inverse function theorem) and thus in $L^2(\Omega,\R^d)$ to $\phi^{-1}_T$.
\\
Second, 
the right-hand side weakly converges in $L^2(\Omega,\R^d)$ to $$\int_0^T \! -[d\phi_{t}]_{\phi_{t}^{-1}(x)}(\partial_t \phi_{t}(\phi_{t}^{-1}(x)))\,\dif t \,.$$ Indeed, let us consider $f \in C^\infty(\Omega,\R^d)$ and calculate the $L^2$ scalar product
\begin{multline}
\langle f, \int_0^T  \! -[d\phi_{n,t}]_{\phi_{n,t}^{-1}(\cdot)}(\partial_t \phi_{n,t}(\phi_{n,t}^{-1}(\cdot)))\,\dif t \rangle = \int_0^T  \! - \langle [d\phi_{n,t}]^*_{\phi_{n,t}^{-1}(\cdot)}(f),\partial_t \phi_{n,t}(\phi_{n,t}^{-1}(\cdot))\rangle\, \dif t  \\
= \int_0^T\! - \langle [d\phi_{n,t}]^{-1*}(f\circ \phi_{n,t}),\partial_t \phi_{n,t}(\cdot) \Jac(\phi_{n,t})\rangle\, \dif t  \,.
\end{multline}
Since $f$ is smooth and $\Omega$ compact, $f$ is uniformly Lipschitz and thus $[d\phi_{n,t}]^{-1*}(f\circ \phi_{n,t})$ converges for the sup norm to $[d\phi_{t}]^{-1*}(f\circ \phi_{t})$. The same convergence holds for $ \Jac(\phi_{n,t})$ by assumption. This proves the weak convergence on smooth functions, which implies the weak convergence in $L^2$ (see \cite{Yosida}).
Strong and weak limits are equal so that 
$[d\phi_{t}]_{\phi_{t}^{-1}(\cdot)}(\partial_t \phi_{t}(\phi_{t}^{-1}(\cdot))) \in L^2([0,1],L^2(\Omega,\R^d))$ is the (time) derivative of $\phi_{t}^{-1}$
and the conclusion ensues.
\end{proof}

\begin{remark}
In fact, we could have proven the following stronger result: the inversion map is continuous on an affine subspace $\tilde{B}$ of $B$ defined by $\tilde{B} = \{ \phi \in B \, | \, \phi_t \in \Diff \,\}$ endowed with the Banach norm $\sup (\| \phi \|_{H^1}, \| \phi \|_{\infty}, \| \phi^{-1} \|_{\infty})$. However, the proof would be a little more involved and the result is not needed in what follows.
\end{remark}

\begin{proposition}\label{AnalyticalProposition}
Solutions in $B$ of \eqref{convel} exist, are unique and are characterized by being solutions of \begin{equation}\label{convelinv2bis}
\partial_t \phi^{-1}(t) = -v(t) \circ \phi^{-1}(t)\,.
\end{equation}
\end{proposition}
\begin{proof}
The initial condition is $\phi_0 = Id$ together with the assumption $\phi \in B$ imply the existence of a positive real number $T>0$ such that $\phi_t$ is a diffeomorphism for $t \in [0,T]$. On this interval, the previous lemma gives that $\phi^{-1} \in B$ and $\partial_t \phi_t = -[d\phi_t]_{\phi_t^{-1}(\cdot)}(\partial_t \phi_t(\phi_t^{-1}(\cdot)))$. Since $\partial_t \phi_t = d\phi(t)\cdot v(t)$, we obtain $\partial_t \phi_t^{-1} = -v(t) \circ \phi^{-1}(t)$. Using  the result \cite[Theorem 8.7]{laurentbook} on flow integration, we obtain the existence and uniqueness of $\phi^{-1} \in B$ satisfying \eqref{convelinv2bis}. This implies also existence and uniqueness of solutions in $B$ of \eqref{convel} on $[0,T]$. 
The extension for all time $t \in [0,1]$ is straightforward by considering $I = \sup \{ T>0 \, | \, \forall t < T \, ,\phi_t \in \Diff \}$.  By construction, $I$ is open and the argument above shows that $I$ is non-empty. Last, $I$ is closed since the flow of $-v(t)$ is a diffeomorphism for all time $t\in [0,1]$ and therefore $I = [0,1]$. 
\end{proof}

\begin{remark}
The definition of the space $B$ could have been a little more general using $W^{1,1}(\Omega,\R^d)$ instead of $H^1(\Omega,\R^d)$.
However, it was not necessary regarding the existence of minimizers of functional \eqref{LDDMM} under \textit{convective velocity} constraint.
\end{remark}

In light of this result, we modify the definitions of $G_L$ and $G_R$ to require that $\phi\in B$:
\begin{align}
G_L &:= \left\{ \phi(1) \in B \, \big| \, \partial_t \phi(t) = d \phi(t) \cdot v(t) \text{ and } v \in L^2([0,1],V) \right\}, \notag\\
G_R &= \left\{ \phi(1) \in B \, \big| \, \partial_t \phi(t) = u(t) \circ \phi^{-1}(t) \text{ and } u \in L^2([0,1],V) \right\}.\notag
\end{align}
Since $G_R$ is closed under inversion, Proposition \ref{AnalyticalProposition} implies $G_L = G_R$.
 Note that the sets of paths $\phi(t)$ in the definitions of $G_L$ and $G_R$ do not coincide in general. 
Indeed, these sets of paths correspond to each other by the inverse map, 
%and as recalled in Proposition \ref{InverseMap}, 
and this inversion shows a loss of regularity for instance on $\Diff^s$.
In the rest of the paper, we will use the notation $G_V$
to denote the group $G_L=G_R$, and by abuse of notation, $G_L$ and $G_R$ will denote the set of paths generated under the constraint \eqref{convel} (and respectively \eqref{spatialvel}) by elements of $L^2([0,1],V)$.  

The structure of $G_V$ is not well-known. 
In the case of Gaussian kernels, $G_V$ is probably included in an ILH-Lie group in the sense of Omori
\cite{MR0431262}. 
In general, it is not known whether $G_V$ admits a differentiable structure.
Nonetheless, the group carries natural left- and right- invariant metrics,
as defined in the next section, and isometries should be understood as being between metric spaces.
In the case of Sobolev spaces, the right-invariant metric is a smooth Riemannian metric, 
whereas the left-invariant metric is probably not.Ê 

\smallskip

Finally, we can now benefit from the existence of minimizers for the functional \eqref{LDDMM} in the LDDMM framework:

\begin{theorem}
If $V$ satisfies assumption \eqref{Admissible} and $E$ is continuous w.r.t. uniform convergence of $\phi$ on every compact set in $\Omega$, then there exists a minimizer in $G_V$ of the functional \eqref{LDDMM} under the \textit{convective velocity} constraint \eqref{convel}.
\end{theorem}

\begin{proof}
This follows from \cite[Theorem 11.2]{laurentbook}.
\end{proof}

Note that the theorem applies for the usual sum of squared differences similarity measure:
$$E(\phi)= \| I \circ \phi(1)^{-1}- J \|_{L^2}^2\,.$$

\section{Left- and right- invariant metrics on diffeomorphism groups}\label{GroupAndCorrespondence}

Proposition \eqref{AnalyticalProposition} proved that the convective velocity constraint \eqref{convel}
%, repeated here,
% \begin{equation}\label{convel2}
%\partial_t \phi(t) = d\phi(t)\cdot v(t)\,,
%\end{equation}
is equivalent to
\begin{equation}\label{convelinv3}
\partial_t \phi^{-1}(t) = -v(t) \circ \phi^{-1}(t)\,,
\end{equation}
in a general setting.This equation is simply the spatial velocity constraint \eqref{spatialvel} for $\phi^{-1}$, except with a minus sign.
In other words, if the spatial and convective velocities of any path $\phi(t)$ are denoted 
by $v_R^\phi$ and $v_L^\phi$,
respectively, then
\begin{equation}\label{E:vLR}
v_R^{\phi^{-1}} = -v_L^\phi.
\end{equation}
As a consequence of this simple fact (well-known in a smooth setting), there are close relationships between Left-LDM and Right-LDM.

\smallskip

On $G_V$ a left-invariant metric $d_L$ can be defined by
\begin{equation}\label{LeftDistance}
d_{L}(\phi,Id) = \inf \{ \sqrt{\int_0^1 \| v^\phi_L(t)\|^2_V \,dt} \, : \, \phi(0) = Id \text{ and } \phi(1) = \phi \}.
\end{equation}
A right-invariant metric $d_R$ can be defined in the same way but using the spatial velocity $v^\phi_R$
instead of the convective velocity $v^\phi_L$.
It follows from \eqref{E:vLR} that
\begin{equation}\label{E:dLR}
d_{L}(\phi,Id) = d_{R}(\phi^{-1},Id).
\end{equation}
 As shown in  \cite{Trouve1998},  the distance $d_R$ is well-defined and makes $G_V$ a complete metric space. From \eqref{E:dLR}, it follows that the same is true of $d_L$.
Between any two diffeomorphisms in $G_V$, there exists a path minimising the distance $d_L$ (\textit{resp}. $d_R$), and such minimising paths will be called left- (\textit{resp.} right-) geodesics.
Note that we have defined geodesics without reference to a Riemannian metric, 
since we do not know whether $G_V$ even has a smooth structure, as discussed earlier.

The following proposition summarises some elementary properties of these distance metrics,
all straightforward consequences of \eqref{E:vLR} and the definitions.

\begin{proposition} \label{InverseMap}
\begin{enumerate}
\item The inverse mapping is an isometry:
\begin{align*}
(G_V,d_L) &\to (G_V,d_R) \\
\phi &\to \phi^{-1}\,
\end{align*}
\item $\phi$ is a left-geodesic if and only if $\phi^{-1}$ is a right-geodesic.
\item Left translation is an isometry of $(G_V, d_L)$, and right translation is an isometry of 
$(G_V, d_L)$.
\item The left translation of a left-geodesic is a left-geodesic (and similarly for right-geodesics).
\end{enumerate}
\end{proposition}

\begin{remark} In the context of fluid dynamics,
$\phi$ is the usual Lagrangian map, and $\phi^{-1}$ is the ``back-to-labels'' map.
Observation (2) in the above proposition has been exploited before in this context \cite{GaRa11Clebsch}.
\end{remark}

We now show two correspondences between Left- and Right- LDM.
% the first one defines the square of a right-invariant distance on the group of diffeomorphisms whereas the second term involves the left action on the source image. This situation can be generalized in the following proposition, where the correspondence between left and right is written explicitly.

\begin{lemma}
%Let $v \in L^2([0,1],V)$ be a time-dependent vector field and $t \to \phi(t)$ its associated "spatial velocity" flow \eqref{spatialvel} with initial condition $\phi(0) = Id$. Then, the "convective velocity" flow associated to $t \to v(1-t)$ with initial condition $\psi(0) = Id$ is given by $\psi: t \to \phi(1) \phi^{-1}(1-t)$.
Let $\phi(t)$ be a path of diffeomorphisms with spatial velocity $v^\phi_R(t)$, defined by $\eqref{spatialvel}$.
Define $\psi: t \to \phi(1) \phi^{-1}(1-t)$, and let $v^\psi_L(t)$ be its convective velocity, defined by 
$\eqref{convel}$. Then $v^\psi_L(t) = v^\phi_R(1-t)$.
\end{lemma}

\begin{proof}
From \eqref{spatialvel} we have, by direct calculation:
\begin{equation}\nonumber
\partial_t \phi^{-1}(t) = -d\phi^{-1}(t) \cdot v_R^\phi(t)\,,
\end{equation}
so that 
\begin{equation}\nonumber
\partial_t \phi^{-1}(1-t) = d\phi^{-1}(1-t)  \cdot  v_R^\phi(1-t)\,,
\end{equation}
and therefore,
\begin{equation}\nonumber
\partial_t \psi(t) =d \psi(t)  \cdot  v_R^\phi(1-t)\,.
\end{equation}
Thus $v_R^\phi(1-t)$ satisfies the relation \eqref{convel} that defines $v^\psi_L(t)$.
\end{proof}

The following proposition is a direct consequence of the previous lemma.
It concerns a generalisation of the matching functional \eqref{LDDMM},
in which the squared path length in the first term is replaced by the integral of 
a general Lagrangian $l(v(t))$, 
and the
image dissimilarity term $E(\phi(1)\cdot I, J)$ is replaced by a general real-valued function $H(\phi(1))$.

\begin{proposition}\label{samephi1}
Let $V$ and $G = G_L = G_R$ be as defined above.
Let $H:G \mapsto \R$ and $l:V \mapsto \R$ be smooth maps. Let $v^\phi_R$ and $v_L^\phi$ be the spatial and convective velocities defined
by \eqref{spatialvel} and \eqref{convel}, respectively. 
We define $\mathcal{F}_R$ on the set of paths in $G_R$ such that $\phi(0) = Id_\Omega$ by  
\begin{equation}
\mathcal{F}_R(\phi(t)) = \int_0^1 \! \ell(v_R^\phi(t)) \, \dif t + H(\phi(1)) \,.
\end{equation}
Respectively, $\mathcal{F}_L(\phi) $ is defined on the set of paths in $G_L$ by
\begin{equation}
\mathcal{F}_L(\phi(t)) = \int_0^1 \! \ell(v_L^\phi(t)) \, \dif t + H(\phi(1)) \,.
\end{equation}
Then,
%Under the assumption that $\ell(t,v) = \ell(1-t,v)$, the following equality holds
\begin{equation}
 \mathcal{F}_R(\phi(t)) = \mathcal{F}_L(\phi(1) \phi^{-1}(1-t)) \,,
\end{equation}
and as a consequence, the minimizers of $\mathcal{F}_R$ and $\mathcal{F}_L$ are in one to one bijection by the map
$\phi(t) \mapsto \phi(1) \phi^{-1}(1-t)$.
\end{proposition}

\begin{proof}
Let $\psi: t \to \phi(1) \phi^{-1}(1-t)$.
Changing the variable $t \mapsto 1-t$, and then applying the Lemma, we have 
\begin{align*}
\int_0^1 \! \ell(v_R^\phi(t)) \, \dif t = \int_0^1 \!  \ell(v_R^\phi(1-t)) \,\dif t 
= \int_0^1  \! \ell(v_L^\psi(t)) \,\dif t .
\end{align*}
Since $\psi(0) = \phi(0) = Id_\Omega$ and $\psi(1) = \phi(1)$, the result follows.
\end{proof}

\begin{remark}
\begin{enumerate}
\item Generically, changing from right- to left- invariant Lagrangian does not change the endpoint of the optimal path.
\item The correspondence also holds for the boundary value problem, \textit{i.e.} when $\phi(1)$ is fixed. 
\item One can use a time-dependent Lagrangian $\ell(v,t)$ if 
%it is symmetric w.r.t. $t=\frac 12$, \textit{i.e.} 
$\ell(v,1-t) = \ell(v,t)$ for all $t \in [0,1]$.
\item If the term $H$ is replaced by a path-dependent term, then the result does not hold any more.
%\item Strictly speaking, the map
%$\phi(t) \mapsto \phi(1) \phi^{-1}(1-t)$ goes from the set of paths generated by $L^2([0,1],V)$ subject to the spatial velocity constraint into the set of paths  generated by $L^2([0,1],V)$ subject to the convective velocity constraints. {\color{blue}A priori, those two sets do not coincide even if one has $G_L = G_R$.}
\end{enumerate}
\end{remark}

A direct application of the previous proposition to the case of the kinetic energy defined by $\ell(v) :=  \frac 12 \|v\|_{V}^2$ and $H(\phi) = E(\phi_1 \cdot I, J)$ gives the following corollary. The existence of minimizers for these functionals is guaranteed by \cite{laurentbook}.

\begin{corollary}\label{EquivMatch}
[Equivalence of Optimal Matches in Left- and Right- LDM]
Consider the problem of minimising
\begin{equation} \label{general-matching}
\mathcal{J}(\phi) = \frac 12 \int_0^1 \! \|v(t)\|_{V}^2 \, \dif t + E(\phi_1 \cdot I, J)\,,
\end{equation}
%or
%\begin{equation} \label{general-matching}
%\mathcal{B}(\phi) = \frac 12 \int_0^1 \|v(t)\|_{V}^2 \, dt + E(I, \phi^{-1}_1 \cdot J)\,,
%\end{equation}
for $\phi_0 = Id_\Omega$, and 
with either constraint
\begin{equation}\label{convel3}
\partial_t \phi_t = d\phi_t \cdot v_t\,  \qquad \textrm{(Left-LDM constraint)}
\end{equation}
or
\begin{equation}\label{spatialvel2}
\partial_t \phi_t = v_t\circ \phi_t\,  \qquad \textrm{(Right-LDM constraint)}.
\end{equation}
Then
\begin{enumerate}
\item
%For a fixed functional, either $\mathcal{A}$ or $\mathcal{B}$, 
The optimal endpoint $\phi_1$ is the same with either constraint.
\item
If $\phi_t$ minimises $\mathcal{J}$ in Left-LDM, then 
$\psi_t := \phi^{-1}_{1-t}\circ \phi_1$
minimises $\mathcal{J}$ in Right-LDM.
\item
If $\psi_t$ minimises $\mathcal{J}$ in Right-LDM, then 
$\phi_t := \psi_1\circ \psi^{-1}_{1-t}$
minimises $\mathcal{J}$ in Left-LDM.
\end{enumerate}
Optimal paths in Left-LDM are left-geodesics, while optimal paths in Right-LDM are right-geodesics.
\end{corollary}

In summary, the optimal diffeomorphism $\phi_1$ is the same in both approaches, 
but there are two optimal paths from $Id$ to $\phi_1$: one left- and one right- geodesic.
These two paths are illustrated in the following diagram.

\vspace{-6pt}

\begin{align*}
\begin{array}{cccccccc}
& & \phi_{t_1} & \rightarrow &\phi_{t_2} & &\\
& \nearrow & & & & \searrow & \\
Id & & & & & & \phi_1 \\
& \searrow & & & & \nearrow & \\
& & \phi^{-1}_{1-t_1} \circ \phi_1 & \rightarrow &\phi^{-1}_{1-t_2} \circ \phi_1 & &
\iffalse
 \\
\\
& & \phi_{t_1}\cdot I & \rightarrow &\phi_{t_2}\cdot I & &\\
& \nearrow & & & & \searrow & \\
I & & & & & & \phi_1 \cdot I = J \\
& \searrow & & & & \nearrow & \\
& & \phi^{-1}_{1-t_1} \cdot J & \rightarrow &\phi^{-1}_{1-t_2} \cdot J & &
\fi
\end{array}
\end{align*}

\medskip

When left- (resp. right-) geodesics act on a image, the resulting paths in shape space are
left- (resp. right-) geodesics.
An example is given in Figure \ref{fig:twopaths}.

\begin{figure}
\begin{center}
\includegraphics[scale=0.35]{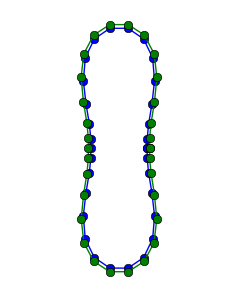}
\includegraphics[scale=0.35]{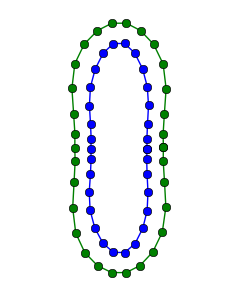}
\includegraphics[scale=0.35]{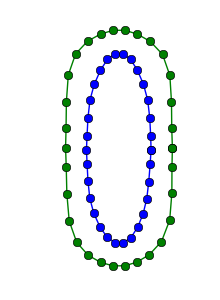}
\includegraphics[scale=0.35]{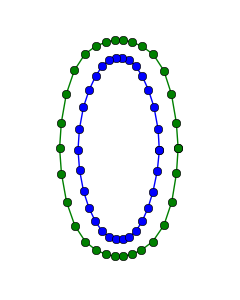}
\includegraphics[scale=0.35]{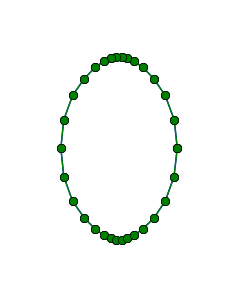}
\end{center}
\caption{
This figure shows snapshots of two deformations from the left-most source image to the right-most target image. The green curves show the optimal Right-LDM path (a right-geodesic), 
while blue curves show the optimal Left-LDM path (a left-geodesic).
Note that the paths are different, though both arrive at an exact match.
The right-metric length of the green geodesic
equals the left-metric length of the blue geodesic.
%(which is shorter than the right-geodesic length of the blue curve).
}
\label{fig:twopaths}
\end{figure}

\section{Geodesic flow of left-invariant metrics}\label{sect:geodesicflow}

We have considered minimisers of \eqref{LDDMM}, which are
geodesics.
We now consider the corresponding initial value problem in which only $\phi(0) = Id$ is fixed.
The minimisers $\phi(t)$ evolve according to Euler-Lagrange equations which are equivalent,
in the Right-LDM case, to the EPDiff equation \cite{HoSchSt09},
\begin{align}\label{EP-right}
\frac{d}{dt} \frac{\partial l}{\partial v} = - \ad^*_{v} \frac{\partial l}{\partial v},
\end{align}
 together with the spatial velocity constraint. 
%(For the definition of $\ad^*$ see \cite{HoSchSt09}.)
This formulation leads to the \textit{momentum representation} of diffeomorphisms,
and further to the special \textit{pulson} solutions, which correspond to image landmarks
and have
applications to optimization schemes \cite{GeodesicShootingVialard,ParticleMesh} 
and to the statistical description of images \cite{Miller09_CFA}.
We now discuss the corresponding concepts in Left-LDM.

%This formulation leads to the \textit{momentum representation} of every final diffeomorphism $\phi(1)$
%by its initial \textit{momentum}, $m(0) := \partial l / \partial v (0)$. 
%Singular \textit{soliton} solutions of \label{EP-right} exist in which the momentum 
%is singular, supported by a finite set of points, and in this case EPDiff reduce to a system of finite-%dimensional ODEs which describe the movements of these points (called ``landmarks'' in 
%image registration) \cite{HolmSolitons}.

%\todo{"solition" or "pulson"?}

The first term of \eqref{LDDMM} with fixed endpoints may be expressed as
$\int_0^1 l(v(t)) dt$ where $l$ is the \textit{kinetic energy Lagrangian} defined by
\begin{equation}\label{kinlagrangian}
l(v) := \frac 12 \|v\|_{V}^2,
\end{equation}
and $v(t)$ is the convective velocity of $\phi(t)$, defined in \eqref{convel}.
The minima of this problem,
with given endpoints $\phi(0)$ and $\phi(1)$, 
are left-geodesics, as defined in the previous section. 
There is a question of the well-posedness of the boundary value problem that defines 
these ``left-geodesics''.
However, from the equivalence with Right-LDM shown in Section \ref{GroupAndCorrespondence},
it follows that the problem \textit{is} well-posed for the same norms for which the 
corresponding problem in Left-LDM is well-posed. In addition, the Euler-Poincar\'e equation is available via this equivalence and let us point out that left-reduction is not needed here.

\medskip
\noindent
\textbf{Euler-Poincar\'e equation.} Using the equivalence with Right-LDM, under mild conditions on $H$ in \eqref{general-matching}, left-geodesics minimising \eqref{general-matching} satisfy the \textit{left} Euler-Poincar\'e equation
\cite{HoSchSt09},
\begin{align}\label{EP-left}
\frac{d}{dt} \frac{\partial l}{\partial v} =  \ad^*_{v} \frac{\partial l}{\partial v}\,.
\end{align}
This equation can be expressed in terms of the \textit{convective momentum},
\[
p(t) := \frac{\partial l}{\partial v}\,,
\]
as $\frac{d}{dt} p =  \ad^*_{v} p$. 
In Euclidean coordinates, the Euler-Poincar\'e equation takes the following form,
called \textit{EPDiff-\textit{left}},  
\begin{align}\label{E:EPDiffL}
\frac{\partial \mathbf{p}}{\partial t}  =  \ad^*_{\mathbf{v}} \mathbf{p}
& := \mathbf{v} \cdot \nabla \mathbf{p} +  \left(\nabla \mathbf{v}\right)^T\cdot \mathbf{p}  + \mathbf{p} \left(\mathrm{div} \, \mathbf{v}\right),
\end{align}
where $\left(\nabla \mathbf{v}\right)^T\cdot \mathbf{p} := \sum_j p_j \nabla v^j$.
%Note that the $i^\textrm{th}$ coordinate of $\mathbf{v} \cdot \nabla \mathbf{p}$ is
%\[
%\left(\mathbf{v} \cdot \nabla \mathbf{p} \right)^i
% := \left(D\mathbf{p} \cdot \mathbf{v} \right)^i
% := \sum_j  \frac{\partial p^i}{\partial x^j} v^j\, .
%\]
If the norm is defined in terms of a kernel $K_\sigma$ as in \eqref{norm}, then
$v= K_\sigma \star p$ and
\begin{equation}
l = \frac 12 \int_0^1  \! \langle p(t), K_\sigma \star p(t)\rangle_{L^2} \, \dif t .
\end{equation}

\medskip
\noindent
\textbf{Conservation law.} Given the convective velocity constraint \eqref{convel}, 
the left-invariant Euler-Poincar\'e equation is equivalent to (see \cite{HoSchSt09})
\begin{align}\label{EP-left-cons}
0 = \frac{d}{dt} \Ad_\phi^* \frac{\partial l}{\partial v} = \left(\phi^{-1}\right)^* p.
\end{align}
This is a conservation law, with the conserved quantity being \textit{spatial momentum},
\begin{align*}
m(t) :=  \left(\phi^{-1}\right)^* p(t) =  \left(\phi^{-1}\right)^* \frac{\partial l}{\partial v}\,.
\end{align*}
Note that this reverses the Right-LDM situation, where convective momentum is preserved
and spatial momentum evolves according to EPDiff-right.

\medskip
\noindent
\textbf{Pulsons.} Singular ``pulson'' solutions may be found by making the following ansatz
\cite{FrHo01},
\begin{equation}
\mathbf{p}(t) = \sum_{a=1}^N \mathbf{P}_a(t) \delta\left(\mathbf{x} - \mathbf{Q}_a(t)\right).
\end{equation}
It is known \cite{HoMa04} that this momentum ansatz defines an 
equivariant momentum map 
\begin{equation}
J_{Sing} : T^*Emb(S, R^n)\to \mathcal{X}(R^n )^* 
\end{equation}
called the \textit{singular solution momentum map},
where here $S$ is a finite set of $N$ points indexed by $a$.
It is the momentum map for the cotangent-lift of the left action of $\Diff(R^n)$ on $Emb(S, R^n)$.
It follows from general theory (see e.g. \cite{HoSchSt09}) that $J_{Sing}$
is a Poisson map with respect to the canonical symplectic form on $T^*Emb(S, R^n)$
and the right Lie-Poisson bracket on $\mathcal{X}(R^n )^*$.
Thus the
EPDiff-\textit{right} equations pull back to canonical Hamiltonian equations in $Q$ and $P$,
with respect to Hamiltonian
\[
H = \sum_{a,b=1}^N \left(\mathbf{P}^a(t) \cdot \mathbf{P}^b(t)\right) 
K\left(\mathbf{Q}^a(t), \mathbf{Q}^b(t)\right).
\]
These are the singular pulson solutions discussed in \cite{FrHo01} and elsewhere.
It also follows, applying a time reversal, that the EPDiff-left equations \eqref{E:EPDiffL}
pull back to time-reversed canonical Hamiltonian equations in $Q$ and $P$, with respect to the
same Hamiltonian:
\begin{align*} 
\frac{\partial }{\partial t}\mathbf{{Q}}_a (t)
&=
- \sum_{b=1}^{N}  \mathbf{P}_b(t)\,
K(\mathbf{Q}_a(t),\mathbf{Q}_b(t))\\
\frac{\partial }{\partial t}\mathbf{P}_a (t)
&=
\sum_{b=1}^{N} 
\big(\mathbf{P}_a(t) \cdot \mathbf{P}_b(t)\big)
\, \frac{\partial }{\partial \mathbf{Q}_a}
K(\mathbf{Q}_a(t),\mathbf{Q}_b(t))\,.
\end{align*}
These are the equations of the pulson solutions of EPDiff-\textit{left}.
Note that they are nearly the same equations as for the pulson solutions of EPDiff-\textit{right}, with two 
important differences: (i) there is a time-reversal; and (ii) $Q_a(t)$ is \textit{not} the spatial location of 
particle $a$ at time $t$, but instead it is an ``anti-particle's location'' in \textit{body coordinates},
i.e. the location in body coordinates corresponding to a fixed spatial location $Q_a(0)$.
This follows from the conservation of spatial momentum.
Similar observations apply to higher-dimensional singular solutions (filaments, sheets, etc.).

%\bigskip
%
%\begin{framed}
%{\color{blue} There should be a discussion on what is lost in this framework. Namely since the metric is changing with the shape, the Riemannian setting is lost.}
%
%To Do: metamorphoses
%
%{\color{blue} There could be a short discussion on right metric + left action are equivalent to left metric + right action. However left metric + right action in continuous time (metamorphoses) does not correspond to standard metamorphoses I guess.}
%
%
%\end{framed}

All of the results in this section can be either verified directly, making minor changes
to the well-known proofs for right-geodesics (the flow of EPDiff-right), or deduced from the correspondence
between left and right geodesics in Section \ref{GroupAndCorrespondence}.

\section{Spatially varying metrics and non-local symmetries}\label{Kernels}
Regarding applications, a crucial point consists in defining the metric which can be viewed as a parameter to be tuned accordingly with data.
In the classical Right-LDM picture, due to translation and rotation symmetry, the class of metrics is rather small. In contrast, the Left-LDM model enables the use of many more types of kernels. In particular, kernels that incorporate non-local correlations. A striking example is the brain development where a symmetry at large scale between the left and the right parts of the brain can be exploited in order to improve the image matching quality. Of course, it is natural to ask for \emph{soft} symmetry in practical applications rather than perfect symmetry. We give hereafter an explicit example of a kernel satisfying those requirements.

Let us first present the case of perfect symmetry: Let $\Pi: V \mapsto V$ be the symmetry of interest, which is a continuous linear operator on the space of vector fields $V$ that satisfies $\Pi^2 = id$. For instance, in $\R^2$ if $v = (v_1,v_2)$, the example showed in the simulation is $\Pi((v_1,v_2))) = (u_1,u_2)$ where $u_1(x,y) = -v_1(-x,y)$ and $u_2(x,y) = v_2(-x,y)$. 
The set of vector fields $v$ satisfying the symmetry condition $\Pi(v) = v$ is thus a closed linear subspace denoted by $V_1$, which may be endowed with the induced norm or alternatively with: 
\begin{equation} \label{NewNorm}
\| v_1 \|_{V_1}^2 = \min_{v \in V} \left\{ \| v \|_V^2 \, \Big| \, \frac 12 (v + \Pi(v)) = v_1\right\}\,. 
\end{equation}
In general, those two norms do not coincide, unless $\Pi$ is self-adjoint which is the case in our example. We prefer the metric \eqref{NewNorm} since the kernel associated with that metric is given by:
\begin{equation}
K_{V_1} = \frac 14 (Id + \Pi)\circ  K \circ (Id + \Pi^*)\,.
\end{equation}
Since, in our example, $\Pi$ is self-adjoint, we can simplify the expression of $K_{V_1}$ to get
\begin{equation}
K_{V_1} = \frac 12 (Id + \Pi)\circ  K\,.
\end{equation}

The above kernel $K_{V_1}$ will produce perfect symmetry which is not desired as mentioned above. 
However, we can modify it to allow for a variable degree of symmetry. 
For example, consider the class of kernels
\begin{equation}
K = (Id + c\Pi)\circ  K_\sigma\,,
\end{equation}
where the strength of the symmetry
ranges from none at $c=0$ to perfect symmetry at $c=1$.
It is also natural to introduce a mixture of kernels with different length scales,
to account for local discrepancies in the deformation field, \textit{i.e.} which means using
\begin{equation} \label{MixedKernel}
K = (Id + c\Pi)\circ  K_{\sigma_1} +  K_{\sigma_2}\,,
\end{equation}
where $\sigma_1,\sigma_2$ are the scale parameters of the kernels, for example the standard deviation of the Gaussian kernel. In particular, it is natural to use $\sigma_1 > \sigma_2$ to account for large scale symmetry.
Looking at the form of the kernel \eqref{MixedKernel}, it is tempting to introduce a spatially varying coefficient that accounts for more or less symmetry or importance of a given kernel.
Therefore, the final example of spatially-varying kernel is the following: Let $K_i$ be $n$ kernels and $\chi_i:\Omega \mapsto [0,1]$ be $n$ smooth functions such that $\sum_{i=1}^n \chi_i = 1$ then we consider
 \begin{equation} \label{MixedKernelchi}
K = \sum_{i=1}^n \chi_i K_i \chi_i\,.
\end{equation}
This kernel is associated to the following variational interpretation:
\begin{equation} \label{NewNorm2}
\| v \|^2 = \min_{(v_1,\ldots,v_n) \in V_1 \times \ldots \times V_n} \left\{ \sum_{i=1}^n \| v_i \|_{V_i}^2 \, \Big| \, \sum_{i=1}^n \chi_i v_i= v\right\}\,. 
\end{equation}
We note that Formula \eqref{NewNorm2} is a simple generalization of mixtures of kernels,
which are explained in detail in \cite{MixtureMMS}.
%In particular, the kernel $K =  a \Pi \circ  K_{\sigma_1} + b \Pi \circ  K_{\sigma_1} + c K_{\sigma_2}$ with $a,b,c$ positive real numbers can be used for practical applications where the symmetry need to be softly imposed on the given scale $\sigma_1$.
%begin{framed}
%Discuss interpretation of spatially-varying metrics.
%Symmetric normalisation

%\section{Generative models of images}
%Given a template $I$, the functional in Equation \eqref{LDD} can be interpreted as a log probability density function $P(v,J)$ on pairs of vector fields $v$ and images $J$.
%This could in theory be marginalised over $v$ to get $P(J)$.
%\begin{align}\label{jointP}
%\log P\left(v,J | I, \lambda,\mathbf{\sigma}\right) 
%&=\log P(v | ,\mathbf{\sigma}) + \log P(J | v, I, \lambda) \\
%&=  \frac 12 \int_0^1 \|v(t)\|_{V_\mathbf{\sigma}}^2 \, dt + \frac{\lambda}{2} \| \phi(1) \cdot I - J\|^2_{L^2}.
%\end{align}
%The parameter $\sigma$ could be a vector of scalar parameters $\sigma_i$, one for each of a set of regions in the template. 
%The parameter $\lambda$, which controls the amount of additive noise in the image, could be a scalar
%or could again be a vector of parameters, one for each region.
%If regional variation is used (in either parameter), then the parameters could be smoothed.
%

%\section{Implementation and experiments}
\section{Experiments}\label{sec:Results}
In the following experiments, we are interested in deformations generated by the Left-LDM model using spatially dependent kernels that incorporate the soft symmetry constraint proposed in Section \ref{Kernels}. By the equivalence proven in Section \ref{GroupAndCorrespondence}, the final deformation is also given by the corresponding Right-LDM model, so all of the numerical results presented below have been computed using the standard gradient descent optimization method for the Right-LDM model detailed in \cite{MixtureMMS}.

%\section{Results}
%\label{sec:Results}

%In this section, we show an example where LIDM with a symmetric metric can help to register two images.
We registered two images out of the LPBA40 dataset  \cite{ShattuckNI2007}. We considered Subjects 8 and 9 of the dataset. The images were resampled to a resolution of 1mm and rigidly aligned. We then extracted corresponding 2D slices from the two aligned images. Finally, we simulated a large lesion in the slice from Subject 8. 
A mask was also constructed, by dilating 
the original lesion location mask 8 times, each time using
a 3$\times$3 structuring element. 
This mask was used to omit lesioned areas from calculation of the
image dissimilarity term, and also
to mask the updated momenta before smoothing.
Registered images are shown in Fig.~\ref{fig:RegisteredImages}.

\begin{figure}[htb!]
\begin{center}
\begin{tabular}{ccc}
\includegraphics[width=0.25\linewidth]{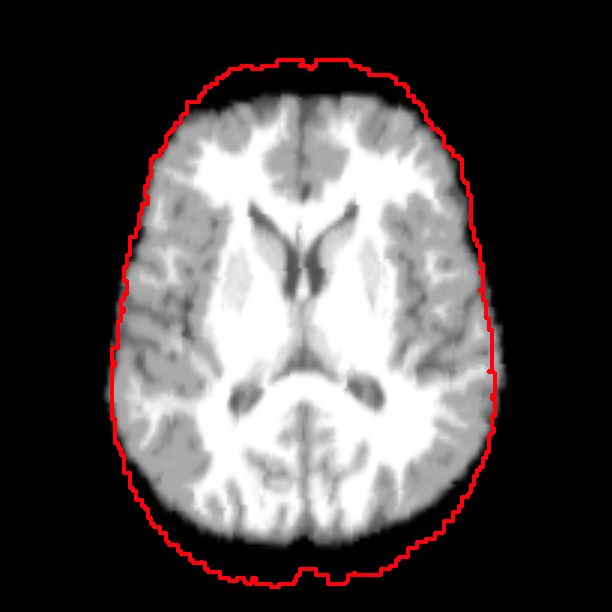}
&
\includegraphics[width=0.25\linewidth]{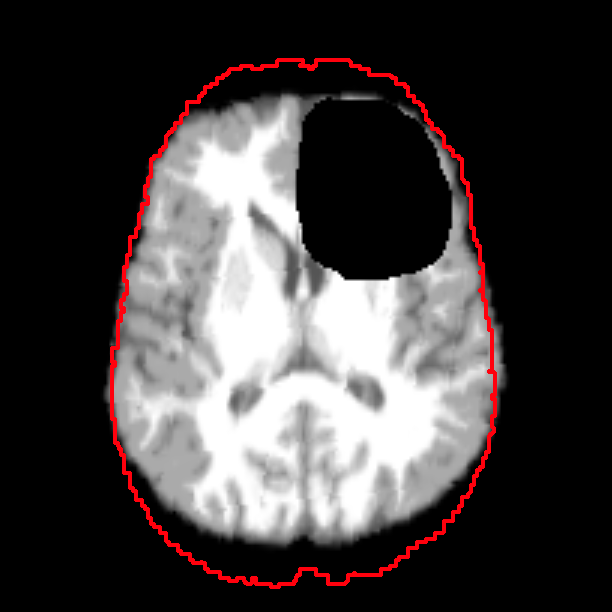}
&
\includegraphics[width= 0.25\linewidth]{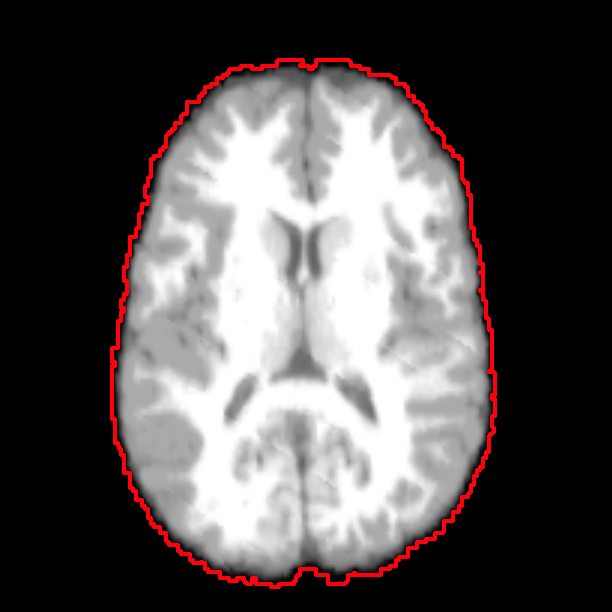}
\\
\end{tabular}
\caption{
%Images registered in the tests of section~\ref{sec:Results} are shown on the top row. 
{\bf (From left to right)}  2D slice from Subject 8 of the LPBA40 dataset; same slice with a simulated lesion (source image); 
and corresponding 2D slice from Subject 9  (target image). 
The red isoline represents the surface of Subject 9's brain.
}
\label{fig:RegisteredImages}
\end{center}
\end{figure}

\begin{figure}[htb!]
\begin{center}
\begin{tabular}{ccc}
% \includegraphics[width=0.25\linewidth]{ImDefLDDMM.png}    
% \begin{picture}(0,0)(0,0)
% \put(-63,65){\textcolor{red}{\huge \bf $\circ$}} 
% \end{picture}
% &
% \includegraphics[width=0.25\linewidth]{ImDefLIDM01.png}    
% \begin{picture}(0,0)(0,0)
% \put(-63,65){\textcolor{red}{\huge \bf $\circ$}} 
% \end{picture}
% &
% \includegraphics[width=0.25\linewidth]{ImDefLIDM10.png}
% \begin{picture}(0,0)(0,0)
% \put(-63,65){\textcolor{red}{\huge \bf $\circ$}} 
% \end{picture}
% \\
\includegraphics[width=0.25\linewidth]{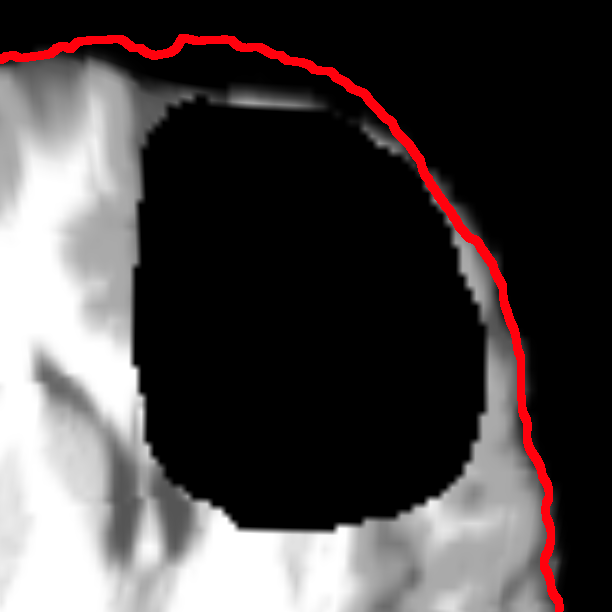}    
\begin{picture}(0,0)(0,0)
\put(-64,78){\textcolor{blue}{\Large $\times$}} 
\end{picture}
&
\includegraphics[width=0.25\linewidth]{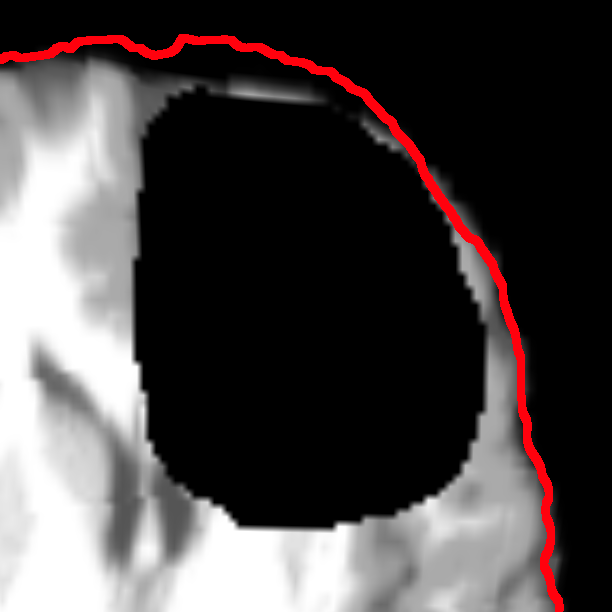}    
\begin{picture}(0,0)(0,0)
\put(-64,78){\textcolor{blue}{\Large $\times$}} 
\end{picture}
&
\includegraphics[width=0.25\linewidth]{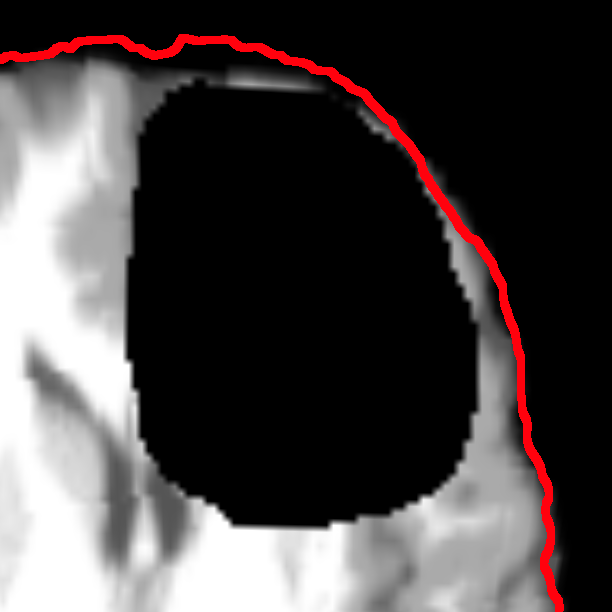}
\begin{picture}(0,0)(0,0)
\put(-64,78){\textcolor{blue}{\Large $\times$}} 
\end{picture}\\
\end{tabular}
\caption{
%Top row contains the deformation grids obtained using the registration strategies 1 and 2 of section \ref{sec:Results} and bottom row contains the corresponding deformed source images around the simulated lesion.  
Source images around the simulated lesion deformed using the registration strategies 1 and 2 of section \ref{sec:Results}.
From left to right, the registration strategies were: non-symmetric kernel (strategy 1);
symmetric kernel (strategy 2) with symmetry weighting factor $c=0.1$;
and symmetric kernel (strategy 2) with symmetry weighting factor $c=1$ (pure symmetry at large scale). 
The red isoline (surface of the target) and the blue cross are always at the same location, to visualize the influence of the symmetric kernel.}
\label{fig:Results}
\end{center}
\end{figure}
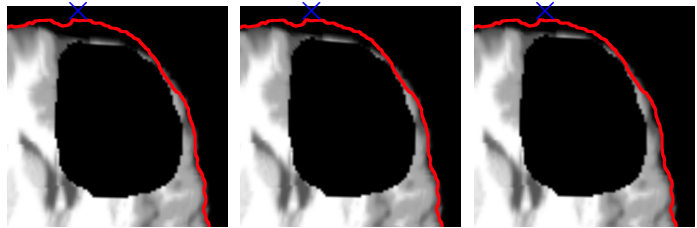

We registered the lesioned images with LDM as described above,
using two kinds of kernel: a standard translationally-invariant sum of Gaussian kernels (non-symmetric);
and a spatially-varying kernel that softly enforces a left-right symmetry:
\begin{enumerate}
\item (non-symmetric)
the sum of two Gaussian kernels
$K_{\sigma_1} + K_{\sigma_2}$, where $\sigma_1 =25\mathrm{mm}$ and $\sigma_2 = 7\mathrm{mm}$,
as in \cite{Begetal2005,Risser11TMI} .
\item
 (symmetric) the sum of a large-scale symmetrised kernel with a small-scale Gaussian kernel,
$K_{\sigma_1} + c \Pi K_{\sigma_1} +  K_{\sigma_2}$,
where $\Pi$ is a reflection about the vertical line dividing the two hemispheres.
The values of $\sigma_1$ and $\sigma_2$ are the same as above, and
$c$ takes values $0.1$ (weak symmetry), $0.5$ or $1.0$ (pure symmetry at large scale).
\end{enumerate}
For comparison, we have also performed LDM registration of the \textit{unlesioned} images using
kernel (1) without a mask.

Deformed images are shown in Fig.~\ref{fig:Results} and deformation magnitudes in the $x$ direction
(horizontal) are shown in  Fig.~\ref{fig:Results2}.
We can see in Fig.~\ref{fig:Results} that modeling a symmetry in the left and right sides of the brain allows 
partial compensation for the information missing in the lesion. The deformations estimated in the lesion are indeed almost only due to the symmetry as clearly emphasized in Fig.~\ref{fig:Results2}. It is also interesting to remark that the most similar deformation to the one obtained without the lesion (image (a) in Fig.~\ref{fig:Results2}) is not the one obtained using pure symmetry on the large scale (image (e) in Fig.~\ref{fig:Results2}), but the one obtained using a factor 0.5 on the symmetry (image (d) in Fig.~\ref{fig:Results2}). In this case, the symmetry plausibly compensates for the missing information at a large scale in the lesion but does not penalize too much the estimation of the deformations in the region symmetric to the lesion.

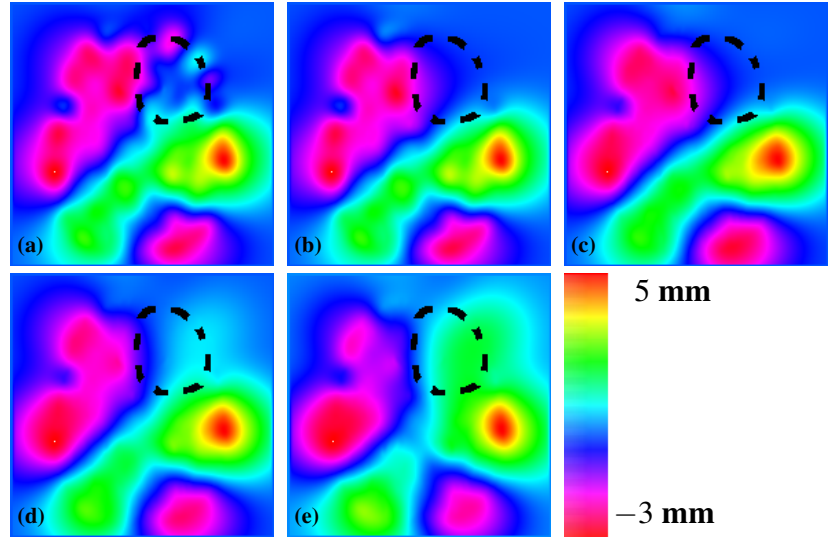
\begin{figure}[htb!]
\begin{center}
\begin{tabular}{ccc}
\includegraphics[height=3.5cm]{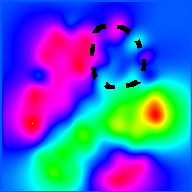}    
\begin{picture}(0,0)(0,0)
\put(-99,5){\bf (a)} 
\end{picture}
&
\includegraphics[height=3.5cm]{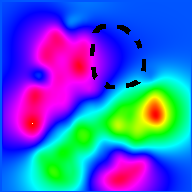}    
\begin{picture}(0,0)(0,0)
\put(-99,5){\bf  (b)} 
\end{picture}
&
\includegraphics[height=3.5cm]{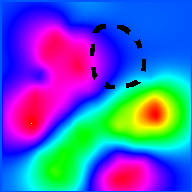}
\begin{picture}(0,0)(0,0)
\put(-99,5){\bf  (c)} 
\end{picture}
\\
\includegraphics[height=3.5cm]{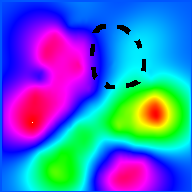}
\begin{picture}(0,0)(0,0)
\put(-99,5){\bf  (d)} 
\end{picture}
&
\includegraphics[height=3.5cm]{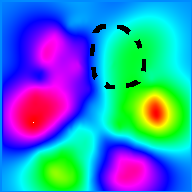}
\begin{picture}(0,0)(0,0)
\put(-99,5){\bf  (e)} 
\end{picture}
&
\includegraphics[height=3.5cm]{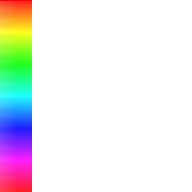}
\begin{picture}(0,0)(0,0)
\put(-75,89){\large \bf $5$ mm} 
\put(-83,5){\large \bf $-3$ mm} 
\end{picture}
\\
\end{tabular}
\caption{
Deformation magnitude in the $x$ direction (horizontal) estimated using the different registration strategies in Section \ref{sec:Results}. 
Results were obtained by registering the images without (a) and with (b-e) the lesion. A mixture of Gaussian kernels was used in (a-b). In (c),(d),(e) a similar mixture of kernels was used, but with a symmetry at the large scale weighted by the factors 0.1, 0.5  and 1, respectively. 
%The color Colorbar is the same for all images. 
The dashed curve represents the boundary of the simulated lesion.
}
\label{fig:Results2}
\end{center}
\end{figure}

\section{Discussion}
We have introduced a new perspective on diffeomorphic image matching,
based on left- (rather than right-) invariant metrics.
For inexact matching with Left-LDM, the optimal diffeomorphism $\phi(1)$ is the same as for
Right-LDM (i.e., the usual LDDMM), however
there are two different optimal paths from the identity to $\phi(1)$ in the diffeomorphism group: one left- and one right- geodesic. This difference could become significant if a 
time-dependent similarity measure were used.

In the Left-LDM setting, it is clear that spatially-varying and nonisotropic kernels describe 
variable deformability properties of the source image.
We have shown, in a numerical experiment, the value of spatially-varying kernels as problem-specific regularisation
terms in inexact matching.
In particular, in a model of a lesioned brain image, we found that a kernel including a large-scale soft symmetry constraint was
successful in compensating for missing information in the lesion area. 

Through the relationship between Left- and Right- LDM, it also becomes apparent that spatially-varying and directionally-dependent kernels in Right-LDM have an interpretation in terms of local 
deformability properties of the source image,
% of the target image, 
which has not been remarked upon in the literature.

One very promising avenue for further work is to replace ad-hoc regularisation choices with 
automatically learnt ones, as has been done by
Simpson et al. \cite{SimpsonNI2012} for global regularisation parameters.
Similar methods could be developed for spatially-varying and directionally-dependent regularisation, 
based on a generative Left-LDM model.
Given a template image $I$, the LDM functional \eqref{LDDMM} can be interpreted as a log probability density function on pairs of initial vector fields $v(0)$ and images $J$:
 \begin{align}\label{jointP}
 \log P\left(v(0),J | I, \lambda,\mathbf{\sigma}\right) 
 &=\log P(v(0) | \mathbf{\sigma}) + \log P(J | v(0), I, \lambda) \\
 &=  \frac 12 \int_0^1 \|v(t)\|_{V_\mathbf{\sigma}}^2 \, dt + \frac{\lambda}{2} \| \phi(1) \cdot I - J\|^2_{L^2},\notag
 \end{align}
 with the constraint \eqref{convel} determining $v(t)$ and $\phi$ from $v(0)$.
 This could in theory be marginalised over $v$ to get $P(J)$.
 Both the regularisation parameters $\sigma$ 
 and the noise parameters $\lambda$ could be spatially-varying, possibly expressed in terms of labels
 associated with the template.
A variety of more or less standard methods could be used to optimise the parameters
 for a population of targets, including Bayesian methods related to those in
  \cite{aat07,cotterBayes}.

\bibliographystyle{spmpsci}      
\bibliography{leftmetrics}    
\end{document}